\documentclass{amsart}
\usepackage{amssymb}
\usepackage{url}
\usepackage{slashbox}
\newtheorem{thm}{Theorem}[section]
\newtheorem{cor}[thm]{Corollary}
\newtheorem{lem}[thm]{Lemma}
\newtheorem{prop}[thm]{Proposition}

\newtheorem{ex}[thm]{Example}

\newtheorem{remark}{Remark}[section]
\newcommand{\dip}{\displaystyle}
\newcommand{\ord}{\mbox{ord}}
\newcommand{\cchi}{\mbox{\raise.48ex\hbox{\,$\chi$}}}
\newcommand{\ceil}[1]{\lceil #1 \rceil}
\newcommand{\N}{\mathbb N}
\newcommand{\Z}{\mathbb Z}
\newcommand{\Q}{\mathbb Q}
\newcommand{\F}{\mathbb F}
\def\quotient#1#2{%
  \raise1ex\hbox{$#1$}\Big/\lower1ex\hbox{$#2$}%
}
\begin{document}
\title{Dynamics of a quasi-quadratic map}

\thanks{This work was partially supported by Funda\c{c}\~ao para a
  Ci\^encia e Tecnologia (FCT), through Centro de Matem\'{a}tica da
  Universidade do Minho and Centro de Matem\'{a}tica da Universidade
  do Porto, FCT project UT-Austin/MAT/0035/2008 and Program POSI}

\author{Assis Azevedo}
\address{Department of Mathematics and  Applications, University of Minho,
  Campus de Gualtar, 4710-057 Braga , Portugal}
\email{assis@math.uminho.pt}

\author{Maria Carvalho}
\address{Department of Mathematics, University of Porto Rua do Campo
  Alegre 687, 4169-007 Porto, Portugal}
\email{mpcarval@fc.up.pt}

\author{Ant\'onio Machiavelo}
\address{Department of Mathematics, University of Porto Rua do Campo
  Alegre 687, 4169-007 Porto, Portugal}
\email{ajmachia@fc.up.pt}

\date{29 September 2012}

\keywords{Discrete dynamical system, ceiling function, density,
  covering system.}

\subjclass[2010]{11A07, 37P99}

\begin{abstract}
  We consider the map $\cchi:\Q\to\Q$ given by $\cchi(x)= x\ceil{x}$,
  where $\ceil{x}$ denotes the smallest integer greater than or equal
  to $x$, and study the problem of finding, for each rational, the
  smallest number of iterations of $\cchi$ that eventually sends it
  into an integer. Given two natural numbers $M$ and $n$, we prove
  that the set of irreducible fractions with denominator $M$ whose
  orbits by $\cchi$ reach an integer in exactly $n$ iterations is a
  disjoint union of congruence classes modulo $M^{n}$, establishing
  along the way a finite procedure to ascertain them. We also describe
  an efficient algorithm to decide if an orbit fails to hit an integer
  until a prescribed number of iterations, and deduce that the
  probability that an orbit enters $\Z$ is equal to one.
\end{abstract}

\maketitle

%%%%%%%%%%%%%%%%%%%%%%%%%%%%%%%%%%%%%%%%%%%%%%%%%%%%%%%
\section{Introduction}

Let $\cchi:\Q\to\Q$ be the map given by $\cchi(x)= x\ceil{x}$, where
$\ceil{x}$ denotes the smallest integer greater than or equal to $x$,
and consider the orbits $\left(\cchi^n(x)\right)_{n\in\N_0}$ of any
$x\in\Q$. We note that $\Z$ is invariant by $\cchi$, the fixed points
are the rational elements in $[0,1]$, $\cchi^{-1}\left(\{0\}\right)=\,
]-1,0]$, $\cchi(-x)=\cchi(x)-x$ if $x\in\Q\setminus\Z$, and that, if
$x\leq -1$, then $\chi(x)\geq1$.

For $\frac{p}{q}\in\Q\,\cap\,[1,+\infty[$, where $p,q$ belong to $\N$
and $(p,q)=1$, the iterate $\cchi^j\left(\frac{p}{q}\right)$ is an
irreducible quotient $\frac{p_j}{q_j}$, where $q_{j+1}$ divides
$q_j$. Therefore the sequence of denominators $q_j$ is decreasing,
although not strictly in general. For instance, the first iterates of
$\frac{31}{10}$ are
\begin{eqnarray*}
  \frac{62}{5},\quad\frac{806}{5},\quad\frac{130572}{5},\quad
  681977556.
\end{eqnarray*}
The number of iterates of $\frac{p}{q}$ needed to hit an integer may
be as large as we want. In fact, if we have any finite sequence of
positive integers $(q_j)_{j\leq n}$ where $q_{i+1}$ divides $q_j$ for
any $j$ then there exists $x\in\Q$ such that
$\cchi^j(x)=\frac{p_j}{q_j}$ with $(p_j,q_j)=1$ (see Remark
\ref{surjectivity}).  However, numerical evidence suggests that, for
any such $\frac{p}{q}$, there is a $j \in \N$ verifying $q_j=1$. This
behaviour bears a resemblance to the dynamics of
$G:\Q\,\cap\,[0,1]\to\Q\,\cap\,[0,1]$,
$G(x)=\frac{1}{x}-[\frac{1}{x}]$, $G(0)=0$, although in this case the
orbit of each rational number in $[0,1]$ is a sequence of irreducible
fractions whose denominators decrease strictly before it ends at $0$,
and this happens in finite time.

For $x\in\Q$, define the \emph{order of $x$} as
$$\ord(x)=\min\{k\in\N_0: \cchi^k(x)\in\Z\}$$
if this set is nonempty, and $\ord(x)=\infty$ otherwise. The integers
are the elements of order $0$; the rational numbers in $]0,1[$ have
infinite order. It is easy to evaluate the order of any irreducible
fraction $\frac{a}{2}$ in $\Q\,\cap\,[1,+\infty[$: given an odd
$a\in\N$, say $a=2^kb+1$ for a positive integer $k$ and an odd $b$,
using induction on $k \in \N$ and the equality
$$\cchi\left(\frac{a}{2}\right)=
\frac{2^{k-1}b\left(2^kb+3\right)+1}{2},$$
one has $\ord\left(\frac{a}{2}\right)=k$.  In particular, for each
$k\in\N$, the smallest positive irreducible fraction with denominator
$2$ whose order is $k$ is $\frac{2^k+1}{2}$. We note that not only
this smallest value increases with $k$, but it does so
exponentially. On the other hand, the following table, which displays
the smallest integer $a$ such that $1 \leq \ord(\frac{a}{3})\leq 50$,
shows that, within the rational numbers with denominator $3$, that no
longer holds.

\begin{figure} [h]
  \centering
  \begin{tabular}{|r|r|| r|r||r|r||}\hline
    order & smallest integer $a$ &order& smallest integer $a$ & order&
    smallest integer $a$\\\hline\hline
    1&7&18&2 215 & 35& 6 335 903 \\
    2&4& 19&6 151 &  36& 1 180 939\\
    3&13& 20&8 653& 37& 1 751 431\\
    4&20& 21&280& 38& 10 970 993\\
    5&10 & 22& 28& 39& 17 545 207\\
    6&5& 23&1 783& 40& 66 269 497\\
    7&29& 24&81 653& 41& 27 952 480\\
    8&76&25&19 310& 42&60 284 614\\
    9&50&26&114 698& 43& 203 071 951\\
    10&452&27&18 716& 44& 191 482 466\\
    11&244&28&196 832& 45& 144 756 173\\
    12&830 & 29&15 214&46& 45 781 445 \\
    13&49 & 30&7 148 &47& 1 343 664 136\\
    14&91& 31&273 223 & 48& 223 084 774\\
    15&319& 32&3 399 188 &49&1 494 753 473\\
    16&2 639& 33& 398 314& 50& 20 110 862\\
    17&5 753& 34 & 6 553 568 &  & \\
    \hline
  \end{tabular}
  \caption{\small Smallest positive integer $a$ such that
    $\frac{a}{3}$ has order between $1$ and $50$.}
\label{fig:caso3}
\end{figure}

We have also verified that, for $a\leq 2\,000\,000\,000$, the order of
$\frac{a}{3}$ is equal or less than $56$. Clearly, this computation
was not achieved directly from the definition of order, because the
iterates grow very rapidly: for example, $\frac{28}{3}$ has order $22$
and $\cchi^{22}(\frac{28}{3})$ is an integer with $4\, 134\, 726$
digits. Our numerical experiments were possible due to two redeeming
features: the dynamical nature of the problem, which allowed us to
reduce the difficulty in each iteration; and, moreover, the location
of the rational numbers with given denominator and a fixed order among
the elements of specific congruence classes modulo a certain power of
the denominator, which enabled us to deal only with numerators that
are limited by that power.

\medskip

In what follows, and after showing that the elements of order $n$ lay
in some congruence classes, we give in Theorem \ref{A(n,M)} a
recursive formula for the number of those classes. We use it to attest
that the natural, or asymptotic, density (see \cite{Ten}, p.~270) of
the elements of $\Q\,\cap\,[1,+\infty[$ that have infinite order is
zero (this follows from Theorem \ref{prob}). We then present an
efficient algorithm to determine if a rational number has an order
below a given bound. Finally we comment on some alternative approaches
and affinities that this problem seems to have with the Collatz
conjecture and the Erd\"os-Straus conjecture on unit fractions.

%%%%%%%%%%%%%%%%%%%%%%%%%%%%%%%%%%%%%%%%%%%%%%%%%%%%%%%
\section{Numbers of order $n$}

We are not aware of any efficient algorithm to evaluate the order of a
rational number. We also do not know if there are numbers, besides the
ones in the interval $]0,1[$, with infinite order. But we do have two
algorithms to decide if a rational $\frac{a}{M}$ has order $n$, for a
fixed $n$. We will see that, in both cases, one only needs to consider
$a<M^{n+1}$.

The first algorithm gives a way to find all the elements of order $n$
if one knows all the elements of order $n-1$, and it is a subproduct
of the results in this section. It relies on the resolution of a
number of quadratic congruences that increases exponentially with $n$,
but what is more significative is that one gets substantial
information on the structure of the elements of order $n$, which is
sufficient to prove that almost all elements have finite order.  The
second algorithm will be presented in section
\ref{sec:an-effic-algor}.

The underlying basic idea is simply to use the obvious fact that
$\ord\left(\cchi(x)\right)=\ord(x)-1$, for each $x\in\Q\setminus\Z$,
to find information about the elements of order $n$ from the ones of
order $n-1$, somehow reversing the dynamics of the map $\cchi$.
Surely, given $y\in\Q$, there is not in general a rational $x$ such
that $\cchi(x)=y$ (consider $y=\frac{5}{3}$, for example). But, in the
proof of Theorem \ref{A(n,M)}, we prove that, in some sense, the
process is reversible.

Let us start by caracterizing the elements that have order equal to
$1$.

\begin{lem}\label{um}
  Let $x=\frac{a}{M}$, where $a\in\Z$, $M\in\N$, with $M>1$, and
  $(a,M)=1$. Then $\ord(x)=1$ if and only if there exists
  $r\in\{1,2,\ldots, M\}$ such that $(r,M)=1$ and $a\equiv
  -r\pmod{M^2}$.
\end{lem}

\begin{proof}
  We note first that $x\not\in\Z$, and also that $x$ has order $1$ if
  and only if $\cchi(x)\in\Z$, which is equivalent, since $(a,M)=1$,
  to the condition that $M$ divides $\ceil{\frac{a}{M}}$. Consider
  $k\in\Z$ and $0< r<M^2$ such that $a=kM^2-r$; notice that, as
  $(a,M)=1$, we have $(r,M)=1$ and $r\neq0$. Then
  $\ceil{\frac{a}{M}}=kM+\ceil{\frac{-r}{M}}$ and so
  \begin{eqnarray*}
    \mbox{$M$ divides $\ceil{\frac{a}{M}}$} & \Leftrightarrow &
    \mbox{$M$ divides $\ceil{-\frac{r}{M}}$} \\
    & \Leftrightarrow & \mbox{$\ceil{-\frac{r}{M}}=0$,\quad as
      $-M<-\frac{r}{M}<0$} \\
    & \Leftrightarrow & \mbox{$r\in\{1,2,\ldots,M-1\}$}.
  \end{eqnarray*}
\end{proof}
This lemma provides the basis for the induction in the proof of the
following result.
\begin{prop}\label{union}
  If $n\in\N$, then, for all $M \in \N$, the set
  \begin{eqnarray}\label{set} {\mathcal A}_{n,M} =
    \left\{a\in\Z:\> (a,M)=1,\> \ord\left(\frac{a}{M}\right)=n\right\}
  \end{eqnarray}
  is a disjoint union of congruence classes modulo $M^{n+1}$.
\end{prop}
\begin{proof}
  If $n=1$, the result is given by the previous Lemma. When $n>1$, we
  only need to guarantee that, if $a\in\Z$, $(a,M)=1$ and
  $\ord\left(\frac{a}{M}\right)=n$, then
  $\ord\left(\frac{a}{M}+tM^{n}\right)=n$, for all $t\in\Z$. Now, if
  $\cchi\left(\frac{a}{M}\right)=\frac{a'}{M'}$, where $a'\in\Z$ is
  such that $(a',M')=1$, and $M'$ is a divisor of $M$, then
  \begin{eqnarray*}
    \cchi\left(\frac{a}{M}+tM^{n}\right) & = &
    \left(\frac{a}{M}+tM^{n}\right)\left\lceil\frac{a}{M}+
      t M^{n}\right\rceil\\
    & = &
    \left(\frac{a}{M}+tM^{n}\right)\left(\left\lceil\frac{a}{M}\right\rceil
      +tM^{n}\right)\\
    & = & \cchi\left(\frac{a}{M}\right)+m M^{n-1},\ \mbox{for some $m\in\Z$}\\
    & = & \frac{a'}{M'}+m M^{n-1}.
  \end{eqnarray*}
  Noting that $\ord\left(\frac{a'}{M'}\right)=n-1$, the result follows
  by induction on $n$.
\end{proof}

From this Proposition we conclude that, when looking for rational
numbers with order $n$ in ${\mathcal A}_{n,M}$, we need only to deal
with irreducible fractions $\frac{a}{M}$ verifying $a\in\{0, 1,
\ldots, M^{n+1}-1\}$. Thus it is now easy to give examples of elements
with order $n$. For instance, it is straightforward to conclude by
induction on $n$ that, if $p$ is an odd prime number and $n\in\N_0$,
then the following numbers have order $n$:
$$\frac{(p-1)p^n+1}{p} \; ;\;
\frac{(-1)^n p^{n}+p-1}{p} ; \; \frac{-(n+1) p^{n}+p^{n-1}+1}{p} \;
\left( \forall_{n\geq 2} \right).
$$

\medskip

Denote by $A(n,M)$ the number of congruence classes modulo $M^{n+1}$
in ${\mathcal A}_{n,M}$ and by $\varphi$ the Euler function. We have
already seen that, for all $n\in\N$,
\begin{eqnarray}\label{casostriviais}
  \left\{
    \begin{array}{l}
      A(0,1)=1, \\
      A(n,1)=0,\\
    \end{array}\right. & \mbox{and, for $M>1$,}
  &
  \left\{
    \begin{array}{l}
      A(0,M)=0  \\
      A(1,M)=\varphi(M)  \mbox{ \footnotesize (by Lemma \ref{um}).}
    \end{array}\right.
\end{eqnarray}

It turns out that the sequence $\left(A(n,M)\right)_{n\in\N_0}$
satisfies a recurrence relation, for all $M>1$, as shown in the
following result.

\begin{thm}
  \label{A(n,M)}
  For $M, n \in\N$, with $M>1$ or $n>1$,
  \begin{eqnarray}\label{sum}
    A(n,M)=\varphi(M)\,\sum_{d\mid M} A(n-1,d)\left(\frac{M}{d}\right)^{n-1}.
  \end{eqnarray}
\end{thm}
\begin{proof}
  For $n=1$, the result is a consequence of
  \eqref{casostriviais}. When $n>1$, we can ignore the divisor $1$ in
  the sum, since $A(n-1,1)=0$. For each divisor $d>1$ of $M$, let
  \begin{equation*}
    Y_d  =  {\mathcal A}_{n-1,d} \,\cap\, [1, M^{n-1}d[.
  \end{equation*}
  Also set
  \begin{eqnarray*}
    W & = &\{c\in\N: (c,M)=1\}\,\cap\, [1, M[\, ,\\[1mm]
    X & = & {\mathcal A}_{n,M}\, \cap\, [1,\varphi(M^{n+1})]\, .
  \end{eqnarray*}
  Observe that $\#W=\varphi(M)$ and, as $n>1$, that
  $Y_1=\emptyset$. Besides, by Proposition~\ref{union}, the set $Y_d$
  has precisely $A(n-1,d)\left(\frac{M}{d}\right)^{n-1}$ elements.

  Consider now the map
  \begin{eqnarray*}
    \begin{array}{rclc}
      \Phi=(\Phi_1,\Phi_2): & X & \longrightarrow &
      \left({\displaystyle \bigcup_{d\mid M}Y_d}\right)\times W
    \end{array}
  \end{eqnarray*}
  defined as follows. Given $a\in X$ such that
  $\cchi\left(\frac{a}{M}\right)=\frac{b}{d}$, where $d$ is a divisor
  of $M$, $b\in\N$ and $(b,d)=1$, take
  \begin{eqnarray*}
    \Phi_1(a)&=&\frac{r}{d}\, , \text{ where $r$ is the remainder of the
    division of $b$ by $M^{n-1}d$},\\
  \Phi_2(a) &=&\text{ the remainder of the division of $a$ by $M$}.
  \end{eqnarray*}
  Notice that $\Phi_1(a)\in Y_d$ because $b\equiv r\pmod{M^{n-1}d}$,
  hence $b\equiv r\pmod{d^n}$, and, as
  $\ord\left(\frac{b}{d}\right)=n-1$, by Proposition~\ref{union} we
  also have $\ord\left(\frac{r}{d}\right)=n-1$.

  \medskip

  To show that $\Phi$ is a bijection, let $d\not=1$ be a divisor of
  $M$ and $(\frac{r}{d},c)\in Y_d\times W$. Then, for $a\in X$,
  \begin{eqnarray*}
    \Phi(a)=\left(\frac{r}{d},c\right) & \Longleftrightarrow &
    \left\{\begin{array}{l}
        \Phi_1(a)=\frac{r}{d}\\
        \Phi_2(a)=c
      \end{array}
    \right.\\
    &\Longleftrightarrow&
    \left\{\begin{array}{l}
        d\,\frac{a}{M}\ceil{\frac{a}{M}}\equiv r\pmod{M^{n-1}d}\\
        \exists y\in\N,\; 1\leq y<M^n:\; a=My+c.
      \end{array}
    \right.
  \end{eqnarray*}
  As $\cchi\left(\frac{a}{M}\right)$ is an irreducible fraction with
  denominator $d$, then $\ceil{\frac{a}{M}}$ (which is equal to $y+1$)
  must be a multiple of $\frac{M}{d}$, and so
  \begin{eqnarray*}
    \Phi(a)=
    \left(\frac{r}{d},c\right)
    & \Longleftrightarrow &
    \left\{\begin{array}{l}
        \exists y\in\N,\; 1\leq y<M^n,\; \exists k\in\N:\;
        y+1=k\frac{M}{d}\wedge a=My+c\\
        \frac{M^2}{d}\,k^2+(c-M)\,k\equiv r\pmod{M^{n-1}d}
      \end{array}
    \right.
  \end{eqnarray*}
  This last quadratic congruence is equivalent to a system of
  congruences of the form
  $$\frac{M^2}{d}\,k^2+(c-M)\,k\equiv r\pmod{p^t},$$
  where $p$ is a prime that divides $M$ with multiplicity
  $t\in\N$. But each of these congruences has a (unique) solution
  (modulo $p^t$) (details in \cite{Vin}). This is due to the fact that
  this congruence reduces modulo $p$ to $c k\equiv r\pmod{p}$, and
  $c\not\equiv 0\pmod{p}$ since $c\in W$; moreover, the formal
  derivative modulo $p$ of the quadratic polynomial on $k$ is $(c-M)$,
  which is not a multiple of $p$.  Using the Chinese Remainder
  Theorem, we conclude that there exists a (unique) $k\in\N$ such that
  $1\leq k\leq M^{n-1}d$ and $\frac{M^2}{d}\,k^2+(c-M)\,k\equiv
  r\pmod{M^{n-1}d}$. Besides, as $(r,M)=1$, we have
  $k<M^{n-1}d$. Therefore, from
  $\frac{a}{M}<\ceil{\frac{a}{M}}=y+1=k\frac{M}{d}$, we deduce that
  $a<k\frac{M^2}{d}<M^{n+1}$.
\end{proof}
As a consequence of the surjectivity of the function $\Phi$ defined in
the proof just presented, we have that:
\begin{remark}\label{surjectivity}
  If $n\in\N$ and $(q_j)_{j\leq n}$ is a finite sequence of positive
  integers where $q_{j+1}$ divides $q_j$ for any $j$, then there
  exists $x\in\Q$ such that $\cchi^j(x)=\frac{p_j}{q_j}$ with
  $(p_j,q_j)=1$.
\end{remark}

In the particular case of $M$ being equal to a power of a prime, we
get a closed formula for $A(n,M)$.
\begin{cor}\label{Anpk}
  Given a prime $p$ and $k,n\in\N$, then
  $A(n,p^k)=\binom{n+k-2}{n-1}\,\left(\varphi(p^k)\right)^n$.
\end{cor}
\begin{proof}
  Firstly recall that the map $\varphi$ verifies
  \begin{equation}
    \label{varphi}
    \varphi(x^{k+1})=x^k\varphi(x), \quad\mbox{for } x,k\in\N.
  \end{equation}
  As mentioned before, $A(1,M)=\varphi(M)$, and therefore the formula
  is valid for $n=1$ and all $k\in\N$. Let us proceed by induction on
  $n$. If $k\in\N$ and $n\geq 1$, then, by Theorem \ref{A(n,M)}, we
  have
  \begin{eqnarray*}
    A(n+1,p^{k}) & = &
    \varphi(p^{k})\,\sum_{d\mid p^{k}}\left(\frac{p^{k}}{d}\right)^{n} A(n,d)\\
    & = & \varphi(p^{k})\,\sum_{i=1}^{k}p^{(k-i)n}
    \binom{n+i-2}{n-1}\varphi(p^i)^{n},\quad\mbox{by induction}\\
    & = & \varphi(p^k)^{n+1}\sum_{i=1}^{k}
    \binom{n+i-2}{n-1},\quad\mbox{ by }\eqref{varphi}\\
    & = & \binom{n+k-1}{n}\varphi(p^k)^{n+1}.
  \end{eqnarray*}
\end{proof}

Using the recurrence formula given by Theorem \ref{A(n,M)}, we have
easily obtained the values of $A(n,M)$, with $1\leq n\leq 5$ and
$M\leq 20$, as displayed in Figure~\ref{fig:tabela}.

\begin{figure}[!h]
  \centering {\scriptsize
    \begin{tabular}{|c||l|l|l|l|l|l|l|l|l|l|l|}\hline
      \backslashbox{n\kern-3em}{\kern-1em M}
      &2&	3&	4     &5	& 6	&7&	8
      &9	&10&11 & 12\\ \hline\hline
      1	&1	&2	&2	&4	&2	&6&	4	&6
      &4	&10	& 4\\ \hline
      2	&1	&4	&8	&16	&18	&36	&48	&72
      &68	&100 & 112	\\ \hline
      3	&1	&8	&24	&64	&86	&216	&384	&648
      &628&	1000 & 1424\\ \hline
      4	&1	&16&	64	&256&	354	&1296&	2560&	5184&
      5060&	10000 & 13952\\ \hline
      5	&1	&32	&160	&1024	&1382&	7776	&15360
      &38880\!\!&	39124\!\!	&100000\!\! &120768\!\!\\ \hline
    \end{tabular}}\vspace{5mm}
  {\scriptsize
    \begin{tabular}{|c||l|l|l|l|l|l|l|l|}\hline
       \backslashbox{n\kern-3em}{\kern-1em M}&	13&	14	&15&	16
      &17&	18&	19&	20\\ \hline
      1			&12	&6	&8&	8	&16	&6
      &18	&8\\ \hline
      2		&	144	&150&	240&	256	&256&	270
      &324	&416\\ \hline
      3		&1728&	2058	&3872&	5120&	4096	&5670
      &5832	&9952\\ \hline
      4	    &	20736&	24774	&52800	&81920&	65536	&93798
      &104976	& 184576\\ \hline
      5		&248832\!&287466\!	&668288\!
      &1146880\!&1048576\!&1396278\!&1889568\!&3048576\!\\ \hline
    \end{tabular}}
  \label{fig:tabela}
  \caption{\small The value of $A(n,M)$ for $1\leq n\leq 5$ and $M\leq20$.}
\end{figure}

The proof of the previous theorem provides an algorithm to explicitly
compute all the congruence classes of all elements $a\in\Z$ such that
$\frac{a}{M}$ has a certain order. This might be used to decide
whether a rational number has a given order $n$, but it has the
drawback that it would require solving a number of congruences that
grows exponentially with $n$.

\medskip

\begin{ex}
  From Corollary \ref{Anpk}, we know that, if $p$ is prime and $n
  \in\N$, then $A(n,p)=(p-1)^n$. This means that the set of
  irreducible fractions $\frac{a}{p}$ with order $n$ is a disjoint
  union of $(p-1)^n$ arithmetic progressions of ratio $p^{n+1}$. For
  instance, with $a\in\Z$,
  \begin{eqnarray*}
    \ord\left(\frac{a}{2}\right)=n &\iff& a \equiv 2^n+1 \pmod{2^{n+1}}\\
    \ord\left(\frac{a}{3}\right)=1&\iff& a \equiv 7 , 8 \pmod{3^2}\\
    \ord\left(\frac{a}{3}\right)=2&\iff& a \equiv 4 , 11 , 14 , 19 \pmod{3^3}\\
    \ord\left(\frac{a}{3}\right)=3&\iff& a \equiv 13 , 22 , 55, 56 , 59 ,
    64 , 74,  77 \pmod{3^4}\\
    \ord\left(\frac{a}{3}\right)=4&\iff& x \equiv 20, 23, 40, 83, 86, 109,
    118, 128, 131, 157, 163, 172,\\
    &&\hphantom{mm} 191, 194, 211,  229 \pmod{3^5}\\
    \ord\left(\frac{a}{5}\right)=1&\iff& a \equiv 21 , 22, 23, 24 \pmod{5^2}\\
    \ord\left(\frac{a}{5}\right)=2&\iff& a \equiv 18, 29, 32, 37, 44, 52, 56,
    58, 66,  78, 86, 92, 101, 109,\\
    &&\hphantom{mm} 113, 114 \pmod{5^3}
  \end{eqnarray*}
\end{ex}
%%%%%%%%%%%%%%%%%%%%%%%%%%%%%%%%%%%%%%%%%%%%%%%%%%%%%%%
\section{Numbers of infinite order}

We could not find any number outside $]0,1[$ with infinite order.
Nevertheless, we were able to show that almost all numbers have finite
order, whose proof is the aim of this section.

Proposition~\ref{union} yields that, given $M\in\N$, $a\in\Z$ with
$(a,M)=1$ and $x=\frac{a}{M}$,
\begin{enumerate}
\item[a)] for any $n\in\N_0$, the probability that $x$ has order $n$
  is
  $\frac{A(n,M)}{\varphi(M^{n+1})}$;\\[1mm]
\item[b)] the probability that $x$ has finite order equals
  ${\dip\sum_{n=0}^\infty}\frac{A(n,M)}{\varphi(M^{n+1})}$.
\end{enumerate}
Let us compute this probability.

\begin{thm}
  \label{prob}
  If $M\in\N$, then the probability that $\frac{a}{M}$, with
  $(a,M)=1$, has finite order is equal to $1$.
\end{thm}
\begin{proof}
  From the meaning of $A(n,M)$ it is clear that the partial sums of
  the series
  $\mathcal{P}(M)=\dip\sum_{n=0}^\infty\frac{A(n,M)}{\varphi(M^{n+1})}$
  are bounded by $1$, and so the series converges, for any given $M$.
  We need to show that its sum is $1$. We will prove this by induction
  on $M$.  Obviously $\mathcal{P}(1)=1$, as $A(0,1)=1$ and $A(n,1)=0$
  if $n\geq1$. Consider $M>1$ and assume that $\mathcal{P}(M')=1$ for
  all $M'<M$. Then
  \begin{eqnarray*}
    \mathcal{P}(M) & = &
    \sum_{n=1}^\infty\frac{A(n,M)}{\varphi(M^{n+1})},\quad \mbox{as
      $A(0,M)=0$}\\
    & = & \frac{1}{M} \sum_{n=1}^\infty \sum_{d\mid M} \varphi(d) \,
    \frac{A(n-1,d)}{\varphi(d^{n})},\quad\mbox{by (\ref{sum}) and
      (\ref{varphi})}\\
    & = & \frac{1}{M}\sum_{d\mid M}\varphi(d)\sum_{n=1}^\infty
    \frac{A(n-1,d)}{\varphi(d^{n})}\\
    & = & \frac{1}{M}\sum_{d\mid M}\varphi(d)\sum_{n=0}^\infty
    \frac{A(n,d)}{\varphi(d^{n+1})}.
  \end{eqnarray*}

  By hypothesis, $\dip\sum_{n=0}^\infty
  \frac{A(n,d)}{\varphi(d^{n+1})}=\mathcal{P}(d)=1$ if $d<M$;
  therefore, using Gauss' Lemma
  $$\sum_{d\mid x}\varphi(d)=x, \quad\mbox{for } x\in\N,$$
  we deduce that
  \begin{eqnarray*}
    \mathcal{P}(M) & = & \frac{1}{M}\left(\sum_{d\mid M}\varphi(d)-
      \varphi(M)+\varphi(M)\mathcal{P}(M)\right)\\
    & = & \frac{1}{M}\left(M- \varphi(M)+\varphi(M)\mathcal{P}(M)\right) \\
    & = & 1-\frac{\varphi(M)}{M}+\frac{\varphi(M)}{M}\mathcal{P}(M)
  \end{eqnarray*}
  and so $\mathcal{P}(M)=1$, as $\varphi(M)<M$.
\end{proof}

\begin{remark} From this theorem it is easy to infer that, for any
  denominator $M$ and $n\in \N$,\\[1mm]
  \begin{enumerate}
  \item ${\displaystyle\lim_{k\rightarrow +\infty} \frac{\#\{1\leq
        a\leq k: a \in {\mathcal A}_{n,M}\}}{k}=
      \frac{A(n,M)}{M^{n+1}}}$.\\[2mm]
  \item For any $M\in\N$, ${\dip \lim_{k\rightarrow +\infty}
      \frac{\#\{1\leq a\leq k:\; \exists n \in \N_0: \; a \in
        {\mathcal A}_{n,M}\}}{k}= \frac{\varphi(M)}{M}}.$\\[2mm]
  \item The density of the integers whose orbits does not reach $\Z$
    is zero.
  \end{enumerate}
\end{remark}

\begin{cor}
  There is no infinite arithmetic progression in $\Q$ whose elements
  have infinite order.
\end{cor}
\begin{proof}
  Any arithmetic progression in $\Q$ contains an arithmetic
  progression of the form $\left(\frac{s+nrM}{M}\right)_{n\in\N}$,
  with $M\in\N$, $s,r\in\Z$ and $(s,M)=1$. In that case, we note that
  $(s+nrM,M)=1$ and that a number of the form $\frac{a}{M}$, with
  $(a,M)=1$, has probability $\frac{1}{2rM}$ of belonging to this
  arithmetic progression. Using Theorem~\ref{prob}, we conclude that
  the arithmetic progression must include elements of finite order.
\end{proof}

For a while we were tempted to believe that if, for a fixed $M\in\N$
and for all $n\in\N_0$, one has $c_n$ disjoint congruence classes
modulo $M^{n+1}$ such that $\dip\sum_{n\geq 0} \frac{c_n}{M^{n+1}}=1$,
then the union of all those classes is all $\Z$, with the possible
exception of a finite set. This, however, is not true, as can be
confirmed by the next example, in which we chose $M=3$ just to
simplify matters, but where we could just as well have taken an
arbitrary $M$.

\begin{ex}
  The idea is to show that one can inductively construct, for each
  $n\in\N_0$ and $k\in\{1, 2, 3, \ldots, 2^{n}\}$, an element
  $x_{n,k}$ in $\{1, 2, \ldots, 3^{n+1}\}$ in such a way that, if
  $(n,k)\neq (m,s)$, then the classes modulo $3^{\min\{m,n\}+1}$ of
  $x_{n,k}$ and $x_{m,s}$ are disjoint. Moreover, one wants to choose
  those elements so that the union of all these congruence classes
  does not contain any number of the set $Y=\{1+3^n:n\in\N_0\}$, for
  example.

  We start with $x_{0,1}=3$. Now, for a given $n\geq 1$, suppose that
  we have already defined $x_{m,k}\in\{1, 2, \ldots, 3^{m+1}\}$, for
  all $m<n$, satisfying the above mentioned conditions. In the set
  $\{1,2\ldots, 3^{n+1}\}$, there are $n+1$ elements of $Y$ and, for
  each $0\leq i\leq n-1$, we have $3^{n-i}$ elements in the class of
  $x_{i,j}$ modulo $3^{i+1}$. We thus have a total of
  $(n+1)+3^n+2\times3^{n-1}+2^2\times3^{n-2}+\cdots+2^{n-1}\times 3$,
  that is $3^{n+1}-3\times 2^n+(n+1)$ elements already ``used''. Then
  $x_{n,k}$, for $1\leq k\leq 2^{n}$, can be selected among the
  remaining $3^{n+1}-\left(3^{n+1}-3\times 2^n+(n+1)\right)$ elements
  of $\{1,2\ldots,3^{n+1}\}$. This is possible since
  $3^{n+1}-\left(3^{n+1}-3\times 2^n+(n+1)\right)=3\times
  2^n-(n+1)\geq 2^n$.

  We therefore obtain $2^n$ classes modulo $3^{n+1}$, for all
  $n\in\N_0$, which are all disjoint and whose complement contains the
  infinite set $Y$.
\end{ex}

This example shows that if indeed it is true that all rational numbers
bigger that $1$ have finite order, as the numerical computations
suggest, then in order to prove it one has to better understand the
relationships among the congruence classes that make up the sets
${\mathcal A}_{n,M}$.

%%%%%%%%%%%%%%%%%%%%%%%%%%%%%%%%%%%%%%%%%%%%%%%%%%%%%%%
\section{An efficient algorithm}
\label{sec:an-effic-algor}

In this section we describe a simple algorithm that verifies if a
rational number has an order less than a fixed bound. It was precisely
this algorithm that allowed us to obtain the results presented in the
previous tables. It runs very quikly, as long as the computer is able
to store the appropriate numbers.

The strategy behind this procedure is the following. Take
$\frac{a}{M}$ and consider the sequence $(a_n)_{n\in\N_0}$ defined by
\begin{eqnarray*}
  a_0 &=& a \\
  a_{n+1} &=& M\cchi\left(\frac{a_n}{M}\right).
\end{eqnarray*}
If we know that $\frac{a}{M}$ has order less or equal to $N$, then,
for $1\leq s<N$, the order of $\frac{a_s}{M}$ is less or equal to
$N-s$. Hence, using Proposition \ref{union}, we can replace $a_s$ by
the remainder of the division of $a_{s}$ by $M^{N+1-s}$. The order of
$\frac{a}{M}$ will then be the first $s$ such that $a_s$ is a multiple
of $M$. We summarize this as follows:

\bigskip

\noindent \textbf{Algorithm}: Given $M\in\N$, $a\in\Z$ and $N\in\N$,
consider $\frac{a}{M}$ and define the sequence $(r_n)_{n\in\N_0}$ as
\begin{eqnarray*}
  r_0 &=& \mbox{ the remainder of the division of $a$ by $M^{N+1}$} \\
  r_{n+1} &=& \mbox{ the remainder of the division of
    $M\cchi\left(\frac{a_n}{M}\right)$ by $M^{N+1-s}$.}
\end{eqnarray*}
Then
$$ \ord\left(\frac{a}{M}\right)=
\left\{\begin{array}{ll} k, & \mbox{if $\;\exists k\leq N:\; M\mid
      r_k$
      and $M\nmid r_s$, for all $s<k$},\\[2mm]
    > N, & \mbox{otherwise.}
  \end{array}\right.
$$
We highlight the fact that this algorithm only needs to deal with
numbers of length less or equal to $M^{N+1}$, and that each step
reduces the bounds involved.

%%%%%%%%%%%%%%%%%%%%%%%%%%%%%%%%%%%%%%%%%%%%%%%%%%%%%%%
\section{Other approaches}

An alternative approach to study the dynamics of the map $\cchi$ would
be to use the finite expansions of the successive numerators in the
bases given by the respective denominators. One of the problems with
this procedure is that $\cchi$ involves a multiplication in which
carries intervene.  In the particular case where the initial point is
a fraction with denominator equal to a prime number $p$, and without
the carries, the dynamics of $\cchi$ would be one of an infinite
dimensional linear cocycle with base space equal to the space of
almost zero sequences on $p$ symbols and fibers over the field
$\F_p$. This sounds already intricate, but the interference of the
carries is a source of additional difficulties, causing $\cchi$ to
resemble a generalized shift with sensitive dependence on initial
conditions \cite{Mo}. This seems to hint, once more, that the question
of determining whether or not there are rational numbers bigger than
$1$ of infinite order may be a hard problem.

\medskip

For each prime $p$ there is a map related to $\cchi$ that can be
defined on the field of $p$-adic numbers as follows. Given
$$x=\sum_{j\geq k} a_j p^j\in\Q_p,$$
where $k\in\Z$ and $a _j \in \{0,1,...,p-1\}$ for all $j$, put
$$\lceil x\rceil_p= 1+\sum_{j\geq 0} a_j p^j.$$ Now let $\cchi_p:
\Q_p\to\Q_p$ be the map defined by $\cchi_p(x)= x \lceil x\rceil_p$,
which clearly coincides with the quasi-quadratic map when restricted
to the rational numbers whose denominator is a power of $p$ and that
are not integers. Note that this map does not agree with $\cchi$ on
the integers, but has the advantage of being continuous.

One can now consider the question of knowing if, for any $p$-adic
number $x$, there exists $n \in \mathbb{N}_0$ such that $\cchi_p^n
(x)$ is a $p$-adic integer. It turns out that, in this case, there are
elements of infinite order other than the rational numbers in
$]0,1[$. This can be seen as follows. By an argument similar to the
one used to prove Theorem~\ref{A(n,M)}, we may verify that, given $k,n
\in \mathbb{N}$, the set
$$\Lambda_{k,n} =\left\{a\in\Z_p:\;
  \cchi_p^n \left(\frac{a}{p^k}\right)\not\in
  \frac{1}{p^{k-1}}\Z_p\right\}$$ is nonempty and a finite union of
congruence classes in $\mathbb{Z}_p$ modulo $p^{(n+1)k}$, and so
compact as well. It is clear that
$\Lambda_{k,n+1}\subset\Lambda_{k,n}$. Moreover, one can show that the
disjoint congruence classes that constitute $\Lambda_{k,n+1}$ are
equally distributed inside the ones that form $\Lambda_{k,n}$.
Therefore, for any fixed $k$, the set $\dip\bigcap_{n\in\N}
\Lambda_{k,n}$ is nonempty, and since it is the intersection of a
nested sequence of compact nonempty sets, each one formed by an
increasing number of balls with decreasing radius, that are scattered
through the balls of the previous set, we see that this intersection
is a perfect subset (with zero Haar measure) of the locally compact
group $\Q_p$. So, it is uncountable, and hence there are elements of
infinite order in $\Q_p$ besides the ones in $\Q\,\cap\, ]0,1[$.

Moreover, as $\cchi_3$ is the polynomial $x^2 + x$ in $\mathbb{Z}_3$,
and it is continuous, the wild nature of the orders we found among the
fractions $\displaystyle \frac{a}{3}$ (see Fig.~\ref{fig:caso3}) may
somehow be related to the chaotic behaviour of $x^2+x$ in
$\mathbb{Z}_3$, as described in \cite{FL}.

%%%%%%%%%%%%%%%%%%%%%%%%%%%%%%%%%%%%%%%%%%%%%%%%%%%%%%%
\section{Final comments and remarks}

Given a fixed integer $M$, suppose one randomly chooses an integer
$a_1$, sets $M_1=M/(M,a_1)$, and then repeats the process by randomly
choosing $a_2$, leeting $M_2=M_1/(M_1,a_2)$, and so on. Note that this
process somehow generalizes what happens to the successive
denominators of $\cchi(\frac{a}{M})$. Given $M$ and $n$, consider now
the question of determining the probability that $M_n=1$. Denote this
probability by $\mathcal{P}(n,M)$. Clearly: $\mathcal{P}(1,M)=1/M$;
$\mathcal{P}(0,M) = 1$ if $M=1$, and $\mathcal{P}(0,M) = 0$ otherwise;
if $p$ is a prime number, then $\mathcal{P}(n,p) = (1-1/p)^{n-1}/p$.
It is easy to show that the numbers $P(n,M)$ satisfy the following
recurrence relation
\begin{equation*}
  \mathcal{P}(n,M)=\sum_{d|M}\frac{\varphi(d)}{M} \mathcal{P}(n-1,d),
\end{equation*}
that is also satisfied by the numbers $A(n,M)/\varphi(M^{n+1})$, as
can easily be shown by induction on $M$ using Theorem~\ref{set}. Thus,
the probability that a rational number $\frac{a}{M}>1$, with
$(a,M)=1$, has order $n$ under $\cchi$ is exactly the same that the
random process just described ends up after $n$ steps, when starting
with $M$. This reveals that the map $\cchi$ behaves, at least in this
respect, just like a random procedure.

\medskip

There are also some curious analogies between our query, if $\cchi$
has no elements of infinite order, and both the Collatz problem and
the Erd\"os-Straus conjecture. The similarities may be only
superficial, but are nevertheless of some interest. In the Collatz
problem, Riho Terras has shown that the set of elements that have a
finite stopping time $n$, a notion analogous to our concept of order,
is a disjoint union of congruence classes (see Theorem 1.2 in
\cite{Ter}). The issue is then whether these cover all integers. This
is an open question, but Terras has also proved that the density of
the integers that do not have finite stopping time is zero. Both these
results are similar to what was shown in the present paper. As for the
Erd\"os-Straus conjecture, William Webb has shown in \cite{Webb} that
the density of the numbers for which the conjecture is false is zero,
using the fact that one can prove its validity for an infinite set of
congruence classes (see Lemma 2 in \cite{Webb}). The connection with
our problem is here less obvious, but in all three cases, ours and
these other two, one could solve the respective conundrum by showing
that a certain system of congruences is a covering of the
integers\footnote{The question of whether one can prove the
  Erd\"os-Straus conjecture by showing its validity on an infinite
  covering system of congruences is unclear, although there are good
  reasons to believe it to be an approach riddled with difficulties:
  see Terrence Tao considerations on this matter in his blog, at
  \url{http://terrytao.wordpress.com/2011/07/07/on-the-number-of-solutions-to-4p-1n_1-1n_2-1n_3}.}.

\medskip

Finally, we mention that our methods also apply to the map $x\mapsto
x\lfloor x\rfloor$, since $x\lfloor x\rfloor=\cchi(-x)$, for any $x
\in \Q \setminus \Z$.

%%%%%%%%%%%%%%%%%%%%%%%%%%%%%%%%%%%%%%%%%%%%%%%%%%%%%%%

%%%%%%%%%%%%%%%%%%%%%%%%%%%%%%%%%%%%%%%%%%%%%%%%%%%%%%%
\end{document}